\documentclass{amsart}
\usepackage{tikz-cd}
\usepackage{amsmath}
\usepackage{amsfonts}
\usepackage{mathtools}
\usepackage{amsmath,amssymb,amsthm,mathrsfs}
\usepackage[utf8]{inputenc}
\usepackage{hyperref}
\usepackage{enumitem}
\usepackage[T1]{fontenc}

\theoremstyle{plain} 
\newtheorem{theorem}{Theorem}
\numberwithin{theorem}{section}
\newtheorem*{theorem*}{Theorem}
\newtheorem{prop}[theorem]{Proposition}
\newtheorem{prop-def}[theorem]{Proposition-Definition}
\newtheorem{lemma}[theorem]{Lemma}
\newtheorem{coro}[theorem]{Corollary}
\newtheorem{conj}[theorem]{Conjecture}
\newtheorem{definition}[theorem]{Definition}
\theoremstyle{definition}

\newtheorem*{example}{Example}
\theoremstyle{remark} 
\newtheorem{remark}[theorem]{Remark}
\theoremstyle{definition} 
\newtheorem*{ack}{Acknowledgment}

\newcommand{\pa}{\partial}

\newcommand{\RM}{\backslash}
\newcommand{\be}{\begin{equation} }
\newcommand{\ee}{\end{equation} }

\usepackage{color}

\newcommand{\bbC}{\mathbb{C}}
\newcommand{\C}{\mathbb{C}} 
\newcommand{\D}{\mathbb{D}}

\renewcommand{\L}{\mathbb{L}}

\renewcommand{\P}{\mathbb{P}}
\newcommand{\Q}{\mathbb{Q}}
\newcommand{\R}{\mathbb{R}}

\newcommand{\Z}{\mathbb{Z}}

\newcommand{\cH}{\mathcal{H}}

\newcommand{\cL}{\mathcal{L}}
\newcommand{\cM}{\mathcal{M}}
\newcommand{\cN}{\mathcal{N}}

\newcommand{\cP}{\mathcal{P}}

\newcommand{\rarr}{\longrightarrow}

\newcommand{\gr}{\textup{gr}}

\newcommand{\sO}{\mathscr{O}}

\newcommand{\sI}{\mathscr{I}}
\newcommand{\sA}{\mathscr{A}}

\newcommand{\shTA}{\mathscr{T}}
\newcommand{\shD}{\mathscr{D}}

\newcommand{\Spec}{\textup{Spec }}

\newcommand{\an}{\textup{an}}

\newcommand{\supp}{\textup{supp}}

\newcommand{\bs}{{\bf s}}
\newcommand{\bw}{{\bf w}}
\newcommand{\ba}{{\bf a}}
\def\bal{{\boldsymbol{\alpha}}}

\newcommand{\bb}{{\bf b}}
\newcommand{\bc}{{\bf c}}
\newcommand{\bee}{{\bf e}}
\newcommand{\bk}{{\bf k}}

\newcommand{\bg}{{\bf g}}

\newcommand{\bff}{{\bf f}}

\newcommand{\rel}{{\textup{rel}}}
\newcommand{\Ann}{{\textup{Ann}}}
\newcommand{\Chr}{{\textup{Ch}^\rel}}
\newcommand{\Ch}{\textup{Ch}}
\newcommand{\CC}{\textup{CC}}
\newcommand{\Ext}{\mathcal Ext}

\title{Bernstein-Sato ideals and hyperplane arrangements}
\author{Lei Wu}
\address{Lei Wu, Department of Mathematics, University of Utah.
155 S 1400 E,  Salt Lake City, UT 84112, USA}
\email{lwu@math.utah.edu}


\begin{document}

\maketitle

\begin{abstract}
    We study the relation between zero loci of Bernstein-Sato ideals and roots of $b$-functions and obtain a criterion to guarantee that roots of $b$-functions of a reducible polynomial are determined by the zero locus of the associated Bernstein-Sato ideal. Applying the criterion together with a result of Maisonobe we prove that the set of roots of the $b$-function of a free hyperplane arrangement is determined by its intersection lattice.
    
    We also study the zero loci of Bernstein-Sato ideals and the associated relative characteristic cycles for arbitrary central hyperplane arrangements. We prove the multivariable $n/d$ conjecture of Budur for complete factorizations of arbitrary hyperplane arrangements, which in turn proves the strong monodromy conjecture for the associated multivariable topological zeta functions.  
\end{abstract}

\section{Introduction}
\subsection{Bernstein-Sato ideals and diagonal specialization}
Let $X$ be a smooth algebraic variety over $\bbC$ (or a complex manifold) and let $\bff=(f_1,f_2,\dots,f_r)$ be an $r$-tuple of regular functions (resp. holomorphic functions) on $X$. 
Then the Bernstein-Sato ideal of $\bff$, denoted by $B_\bff$, is the ideal in $\bbC[s_1,\dots,s_r]$ generated by $b(\bs)$ satisfying 
\be\label{eq:feqbsi}
b(\bs)\bff^\bs=P\cdot \bff^{\bs+\mathbf 1_r}\ee
for $P\in \shD_X[\bs]=\shD_X\otimes_\bbC \bbC[\bs]$,  where $\bs=(s_1,s_2,\dots,s_r)$, $\bff^\bs=\prod_i f_i^{s_i}$ and $\mathbf1_r=(\underbrace{1,1,\dots,1}_r)$. Or equivalently, the Bernstein-Sato ideal of $\bff$ is the ideal of $\bbC[\bs]$-module annihilators for $\shD_X[\bs]\bff^{\bs}/\shD_X[\bs]\bff^{\bs+\mathbf1_r}$ (see \S\ref{sec:rr}). When $r=1$, the generator of the Bernstein-Sato ideal in $\bbC[s]$, denoted by $b_f(s)$, is called the Bernstein-Sato polynomial (or $b$-function) of the regular function (since $\bbC[s]$ is a PID).

We consider the diagonal embedding 
\[\delta\colon \bbC\hookrightarrow \bbC^r \quad s\mapsto (s,s,\dots,s).\]
Using the equation \eqref{eq:feqbsi}, we have that if $b(\bs)\in B_\bff$, then $b(s,s,\dots,s)\in B_f$, where $f=\prod_{i=1}^rf_i$, which gives an inclusion of zero loci,
\be \label{eq:keyinc}
Z(B_f)\subseteq\delta^{-1}(Z(B_\bff)).
\ee
This process is called the specialization of Bernstein-Sato ideals, see \cite[\S 4.21]{Budur} for more information.

Our main theorem is a criterion to ensure \eqref{eq:keyinc} being an equality.
\begin{theorem}\label{thm:maindiasp}
With notations as above, if $\shD_X[\bs]\bff^\bs$ is Cohen-Macaulay over $\shD_X[\bs]$ $($cf. Definition \ref{def:purecm}$)$, then 
\[\delta^{-1}(Z(B_\bff))= Z(B_f).\]
\end{theorem}


\subsection{Hyperplane arrangements}
A {\em hyperplane arrangement} $D$ is a finite collection of hyperplanes in $X=\bbC^n$, that is, 
\[D=\{D_1,D_2,\dots, D_p\}\]
where $D_j$ are hyperplanes in $X$. Let $f_j$ be the linear polynomial defining $D_j$ and $f=\prod_{j=1}^p f_j^{m_j}$ with $m_j\ge 1$ so that the support of the divisor $(f=0)$ is $D$. We also call $f$ a hyperplane arrangement. We set $f_D=\prod_j f_j$.
We write by $L(D)$ the intersection lattice of a hyperplane arrangement $D$, that is, 
\[L(H)=\{\bigcap_{H\in B}H \mid B\subseteq D\}.\]
Elements in $L(D)$ are called \emph{edges}. For $W\in L(D)$, we write 
$$D_W=\{H\in D\mid W\subseteq H\}, D^W=\{H/W\mid H\in D_W\}$$
and $J(W)=\{j\in\{1,2,\dots,p\}\mid W\subset D_j\}.$ By definition, $D^W$ is a hyperplane arrangement in the quotient space $X/W$. 
The rank of $W$, denoted by rank$(W)$, is the codimension of $W$ in $X$. We set 
\[R_W=\{-\dfrac{\textup{rank}(D_W)}{|J(W)|},-\dfrac{\textup{rank}(D_W)+1}{|J(W)|},\dots,-\dfrac{2J(W)-\textup{rank}(D_W)}{|J(W)|}\}.\]

A hyperplane arrangement $D$ is {\em central} if $f_D$ is homogeneous and it is {\em essential} if $\{0\}\in L(D)$. A central hyperplane arrangement is {\em irreducible} if there is no linear change of coordinates on $\C^n$ such that $f$ can be written as the product of two non-constant polynomials in disjoint sets of variables. An edge $W\in L(D)$ is called \emph{dense} if $D_W$ is irreducible. A central hyperplane arrangement $D$ is \emph{free} if the underlying divisor is \emph{free} in the sense of K. Saito \cite{KSaito} (cf. Definition \ref{def:freediv}).

It is proved by Walther \cite{Wal} that the $b$-function of a hyperplane arrangement is not governed by its intersection lattice. By using Theorem \ref{thm:maindiasp} as well as Maisonobe's formula for the generator of the Bernstein-Sato ideal of a free hyperplane arrangement (see Theorem \ref{thm:mainBffha}), we obtain:  
\begin{theorem}\label{thm:mainrootoffha}
With notations as above, if $D$ is a free hyperplane arrangement, then 
\[Z(B_{f_D})=\bigcup_{\textup{dense } W \in L(D)} R_W,\]
where the union goes over all dense edges in $L(D)$.
\end{theorem}
Recently, Bath obtained the same formula in \cite[Theorem 1.4]{Bath} to calculate roots of $b$-functions for free hyperplane arrangements $D$ but from a different method. In \emph{loc. cit.} Bath also generalized Maisonobe's formula to possibly non-reduced free hyperplane arrangements. One can also generalize Theorem \ref{thm:mainrootoffha} for possibly non-reduced free hyperplane arrangements by using Theorem \ref{thm:maindiasp}; see \cite[Eq.(4.24)]{Bath}.

The above theorem says nothing about the multiplicities. However, M. Saito \cite{SaitoHyp} proved that the multiplicity of the root $s=-1$ of the $b$-function for a central essential hyperplane arrangement is $n$.  Furthermore, Terao's famous conjecture predicts that the freeness of a hyperplane arrangement is determined from its intersection lattice. Motivated by Theorem \ref{thm:mainrootoffha}, we conjecture the same phenomenon happening for $b$-functions of free hyperplane arrangements.
\begin{conj}
The $b$-function of a free hyperplane arrangement is determined from its intersection lattice.
\end{conj}

Suppose that $f$ is a central hyperplane arrangement, not necessarily reduced. An $r$-tuple $\bff=(f_1,f_2,\dots,f_r)$ is called a factorization of $f$ if each $f_j$ is a non-empty central hyperplane arrangement and $\prod_{j=1}^rf_j=f$. It is called a complete factorization if moreover each $f_j$ is linear (hence deg$f=r$). For a complete factorization $\bff=(f_1,f_2,\dots,f_r)$, we set 
\[J(W,\bff)=\{j\in \{1,2,\dots,r\}\mid W\subset(f_j=0)\}.\]
Our next result is about the zero loci of Bernstein-Sato ideals for arbitrary central hyperplane arrangements.
\begin{theorem}\label{thm:maingndcf}
Suppose that $\bff=(f_1,f_2,\dots,f_r)$ is a complete factorization of a central hyperplane arrangement $f$ $($possibly nonreduced$)$. 
If $W$ is a dense edge in its intersection lattice, then 
\[(\sum_{j\in J(W,\bff)}s_j+\textup{rank}(W)+k=0)\subseteq Z(B_\bff).\]
for all $k=0,1,\dots, |J(W,\bff)|-1$
\end{theorem}
The above theorem particularly proves the multivariable $n/d$-conjecture (a natural generalization of the $n/d$-Conjecture of Budur, Musta\c t\u a and Teitler \cite{BMT})
for complete factorizations of central essential irreducible hyperplane arrangements (see \cite[Conjecture 1.13]{Budur}).  

\begin{example}
Take $\bff=(x,y,z,x+y+z)$
and $f=xyz(x+y+z)$ in $\bbC[x,y,z]$ (notice that $f$ is not free). One can compute with dmodideal.lib \cite{DGPS} and get that
$B_\bff$ is principal and generated by \[\prod_{i=1}^4(s_i+1)\prod_{j=0}^3(s_1+s_2+s_3+s_4+3+j).\]
Then Theorem \ref{thm:maingndcf} is optimally verified. 
\end{example}

For an $r$-tuple $\bff$, the modules  $\shD_X[\bs]\bff^\bs$ and $\shD_X[\bs]\bff^\bs/\shD_X[\bs]\bff^{\bs+\mathbf1_r}$ are both relative $\shD$-modules over $\bbC[\bs]$. Maisonobe \cite{Mai} proved that they are indeed relative holonomic (see \S\ref{subsec:relhol} for definition). For relative holonomic $\shD$-modules, their zero loci of Bernstein-Sato ideals are related to their relative characteristic varieties (cf. Lemma \ref{lm:suppzbf}). When $\bff$ is a complete factorization of a central hyperplane arrangement, we obtain the following:
\begin{theorem}\label{thm:mainmchr}
Suppose that $\bff=(f_1,f_2,\dots,f_r)$ is a complete factorization of a central hyperplane arrangement $f$ (possibly nonreduced) and that $W$ is a dense edge in its intersection lattice. Then for every $l\in\Z$, the subvariety of $T^*X\times \bbC^r$, 
\[T^*_W X\times (\sum_{j\in J(W,\bff)}s_j+l=0)\]
is a component of $\Ch^\rel(\shD_X[\bs]\bff^{\bs-\bk_r}/\shD_X[\bs]\bff^{\bs+\bk_r})$ for every $k\gg l$, where $\bk_r=(\underbrace{k,k,\dots,k}_r)$ and $T^*_WX$ is the conormal bundle of $W$ in the cotangent bundle $T^*X$. Moreover, the multiplicity of $T^*_W X\times (\sum_{j\in J(W,\bff)}s_j+l=0)$ is 
\[(-1)^{\textup{rank}(W)-1}\cdot\chi(\P(X/W)\setminus \bigcup_{H\in D^W}\P(H)),\]
where $\chi(\bullet)$ denotes the topological Euler characteristic. In particular, 
\[(-1)^{\textup{rank}(W)-1}\cdot\chi(\P(X/W)\setminus \bigcup_{H\in D^W}\P(H))>0.\]
\end{theorem}
It is well known that $W$ is a dense edge if and only if 
$$\chi(\P(X/W)\setminus \bigcup_{H\in D^W}\P(H))\not=0,$$
see \cite[Proposition 2.6]{STV}.
The above theorem gives a geometric interpretation of this nonzero number. But,
the positivity result in the above Theorem can also be deduced directly from \cite[Corollary 1.4]{FK} (see also \cite[Theorem 1.7]{WZ}), as $D^W$ is a central essential hyperplane arrangement in $X/W$.

\subsection{Multivariable Monodromy Conjecture}
We recall the construction of the topological zeta function. See \cite{Budur} for the discussion of related topics.

Let $\bff=(f_1,f_2,\dots,f_r)$ be an $r$-tuple of polynomials in $\bbC[x_1,x_2,\dots,x_n]$. Denote $X = \C^n$,  $D_i=(f_i=0)$ and $D=\sum_{i=1}^rD_i$. Take a log resolution 
$$\mu\colon (Y, F)\to (X, D),$$ 
i.e. $F$ has support of normal crossing divisors, with irreducible components $E_i$ for $i \in S$. Define integers $a_{ij}, k_i$ by
\[ \mu^* D_j = \sum_i a_{ij} E_i, \quad K_{Y/X} = \sum_i k_i E_i, \]
where $K_{Y/X}$ is the canonical divisor of $\mu$.
The \emph{topological zeta function} of $\bff$ is 
\[Z_\bff(\bs)\coloneqq \sum_{I\subseteq S} \chi(E^o_I)\prod_{i\in I}\dfrac{1}{\sum_{j=1}^ra_{ij}s_j+k_i+1}\]
where $E^o_I=\cap_{i\in I}E_i\setminus \cup_{i\in S \RM I}E_i$.
The rational function $Z_\bff(\bs)$ is independent of the choice of log resolutions. We denote by $\textup{PL}(Z_\bff(\bs))$ the pole locus of $Z_\bff(\bs)$.
\begin{conj}\cite[Conjecture 1.17]{Budur}\label{conj:tmc}
\[\textup{PL}(Z_\bff(\bs))\subseteq Z(B_\bff).\]
\end{conj}
When $r=1$, the above conjecture is the Strong Monodromy Conjecture of Igusa-Denef-Loeser. It (even in the case $r=1$) is widely open. 

As a main application, we obtain:
\begin{theorem}\label{thm:tmccfhyp}
Suppose that $\bff=(f_1,f_2,\dots,f_r)$ is a complete factorization of a hyperplane arrangement $f$ (possibly nonreduced). Then Conjecture \ref{conj:tmc} holds for $\bff$.
\end{theorem}
Since successive blowups along dense edges give a log resolution of a hyperplane arrangement \cite[Theorem 3.1]{STV}, the above theorem is a direct consequence of Theorem \ref{thm:maingndcf} (or one can apply \cite[Theorem 1.18]{Budur}). 

When $f$ is a tame hyperplane arrangement, Walther \cite{Wal} proved  Conjecture \ref{conj:tmc} for $f$ by proving the $n/d$-Conjecture of Budur, Musta\c t\u a and Teitler; based on Walther's idea, Bath \cite{Bath} further proved Conjecture \ref{conj:tmc} for certain factorizations of $f$.

At this moment, we cannot conclude Conjecture \ref{conj:tmc} for the hyperplane arrangement $f$ (namely $r=1$) from Theorem \ref{thm:tmccfhyp} because we cannot exclude the possibility
\[\delta^{-1}(Z(B_\bff))\supsetneq Z(B_f)\]
where $\delta$ is the diagonal embedding as in Theorem \ref{thm:maindiasp}. Conversely, it is not known whether Conjecture \ref{conj:tmc} for $f$ implies Theorem \ref{thm:tmccfhyp}.

\subsection{}
In Section \ref{sec:2}, we recall the general theory of relative $\shD$-modules. Section \ref{sec:rr} is about Bernstein-Sato ideals and their relations with relative characteristic varieties. In Section \ref{sec:spe}, we discuss the diagonal specialization and prove Theorem \ref{thm:maindiasp}. Section \ref{sec:apptohyparr} is devoted to applications for hyperplane arrangements.

\begin{ack}
We are grateful to Mircea Musta\c t\u a for pointing out a mistake in an early version of this paper. We specially thank Peng Zhou for many useful discussions. We thank Takuro Abe for answering questions. We thank Daniel Bath, Nero Budur, Mihnea Popa, Robin van der Veer and Uli Walther for comments and suggestions. 
\end{ack}

\section{Relative $\shD$-modules and characteristic cycles}\label{sec:2}
\subsection{}\label{subsec:relch}
Suppose that $X$ is a smooth algebraic variety over $\bbC$ (or a complex manifold) of dimension $n$ and let $\shD_X$ be the sheaf of rings of the algebraic (or analytic) differential operators on $X$. We let $R$ be a commutative noetherian $\bbC$-algebra integral domain such that the localization at every prime ideal is a regular local ring. We always assume that the Krull dimension of $R$ is finite.
We then set 
\[\sA_R=\shD_X\otimes_\bbC R.\]
For instance, if $R=\bbC[\bs]=\bbC[s_1,\dots,s_r]$ for some integer $r\ge0$, then 
\[\sA_{\bbC[\bs]}=\shD_X[\bs]\coloneqq \shD_X\otimes_\bbC \bbC[s_1,\dots,s_r].\]
A (left or right) $\sA_R$-module $\cM$ is called coherent over $\sA_R$ or relative coherent over $R$ (or $\Spec R$) if $\cM$ is locally finite presented over $\sA_R$. The definition of relative characteristic cycles is similar to that of $\shD$-modules, that is, the case when $R=\bbC$; see \cite[\S 2]{HTT} for this case. We recall the construction in general for completeness; see also \cite[\S 3]{BVWZ}.

The order filtration $F_\bullet$ on $\shD_X$ induces a relative filtration on $\sA_R$ by
\[F^\rel_\bullet\sA_R=F_\bullet\shD_X\otimes_\bbC R, \]
in other words, elements in $R$ have degree zero. Now we assume that $\cM$ is a coherent $\sA_R$-module.
Then a relative good filtration of $\cM$ is a filtration $F_\bullet^\rel$ on $\cM$ compatible with $F^\rel_\bullet\sA_R$ so that 
\[\gr^\rel_\bullet \cM \textup{ is coherent over } \gr^\rel_\bullet\sA_R\simeq \gr_\bullet\shD_X\otimes_\bbC R.\]
Similar to the case for coherent $\shD$-modules, good filtrations always exist in the algebraic case and locally in the analytic case for relative coherent $\shD$-modules.

Since $\Spec(\gr^\rel_\bullet\shD_X[\bs])\simeq T^*X\times \Spec R$, the $\sim$-functor gives that $\widetilde\gr^\rel_\bullet \cM$ a coherent $\sO_{T^*X\times \bbC^r}$-module. We then define the relative characteristic cycle of $\cM$ inside $T^*X\times \Spec R$ to be the cycle of the sum of irreducible components of the support of $\widetilde\gr^\rel_\bullet \cM$ with multiplicities, denoted by $\CC^\rel(\cM).$ We also write the relative characteristic variety of $\cM$, that is the support of $\CC^\rel(\cM)$, by $\Chr(\cM)$. Similar to the case for $\shD$-modules (namely $R=\bbC$), one can check that $\Chr(\cM)$ and $\CC^\rel(\cM)$ are independent of the choices of good filtration. In particular, $m_p(\cM)$ is independent of such choices, where $p$ is the generic point of an irreducible component of $\Chr(\cM)$ and $m_p$ is the multiplicity of $\gr^\rel_\bullet\cM$ at $p$.
In the analytic case, $\CC^\rel(\cM)$ is a locally finite sum.

We denote by $\CC^\rel_{k}(\cM)$ the part of $\CC^\rel(\cM)$ with pure dimension $k$ and by $\Ch^\rel_{k}(\cM)$ the part of $\Ch^\rel(\cM)$ with pure dimension $k$. If $R$ is $\bbC$ or a field extension of $\bbC$, then we use $\Ch$ and $\CC$ to denote the characteristic variety and the characteristic cycle respectively for $\shD$-modules on smooth varieties over the base field $R$.

It is worth mentioning that $\supp(\widetilde\gr^\rel_\bullet\cM)$ might has embedded associated primes for some good filtration and the embedded associated primes depend on filtration. However, for pure modules, there at least exists one good filtration so that $\supp(\widetilde\gr^\rel_\bullet\cM)$ is equi-dimensional without embedded associated primes; see \cite[Theorem A:IV 4.11]{Bj}. 

\subsection{Localization of relative $\shD$-modules}\label{subsection:locrdm}
We discuss localization of relative $\shD$-modules as $R$-modules and localization of relative characteristic cycles. 

Suppose that $\cM$ is a coherent $\sA_R$-module. Let $q\subseteq \sA_R$ be a prime ideal. Then the localization of $\cM$ at $q$ is defined as
\[\cM_q=\cM\otimes_{R} R_q\]
where $R_q$ is the localization of $R$ at $q$. Then $\cM_q$ is a coherent $\sA_{R_q}$-module. 

Define 
\[B_{\cM}\coloneqq\Ann_{R}(\cM)\subseteq R\]
the ideal of annihilators of $\cM$ as an $R$-module, called the Bernstein-Sato ideal of $\cM$ over $R$. By definition, $B_\cM$ localizes, that is, if $S\subseteq R$ is a multiplicative set, then 
\[
S^{-1}B_\cM= \Ann_{S^{-1}R}(S^{-1}\cM)\subseteq S^{-1}R.
\]
In particular, if $q\subseteq R$ a prime ideal, then 
\be\label{eq:localbf}
B_{\cM,q}= \Ann_{R_q}(\cM_{q} )\subseteq R_q.
\ee

Recall that the support of $\cM$ as a $R$-module is 
\[\supp_{R}(\cM)=\{\textup{ prime ideals $m \in \Spec R$ }| \cM_m\not=0\}.\]
By abuse of notations, we also use  $\supp_{R}(\cM)$ to denote the set of closed points of maximal ideal in $\supp_{R}(\cM)$ in $\Spec R$. 
Since $B_\cM$ localizes, we immediately have
\[Z(B_\cM)\subseteq\supp_{R}(\cM),\]
where $Z(B_\cM)$ denotes the zero locus of the ideal $B_{\cM}$.
However, since $\cM$ might not be finite generated over $R$, in general we do not have $Z(B_\cM)=\supp_{R}(\cM)$.

We write the relative characteristic cycle of $\cM$ by
\[\CC^\rel(\cM)=\sum_p m_p\bar p,\]
where $p$ goes over the generic points of the irreducible components of the support of $\gr^\rel_\bullet\cM$ and $\bar p$ is its closure in $T^*X\times \Spec R$. We also define the localization of the characteristic cycle by 
\[\CC^\rel(\cM)_q=\sum_{p\subseteq p_2^*q} m_p\bar p\]
where $p_2\colon T^*X\times \Spec R\to \Spec R$ and $\bar p$ is the closure of $p$ inside $T^*X\times \Spec R_q$.

We take a good filtration $F_\bullet \cM$ and then define
$$F_\bullet \cM_q\coloneqq (F_\bullet\cM)_q\coloneqq (F_\bullet \cM)\otimes_{R}R_q,$$  
which is a good filtration on $\cM_p$ as the localization functor is exact.
We then immediately have 
\[\gr^\rel_\bullet (\cM_q)=\gr^\rel_\bullet(\cM)_q\coloneqq \gr^\rel_\bullet(\cM)\otimes_{R}R_q.\]
As a consequence, we have 
\be\label{eq:ccloc}
\CC^\rel(\cM_q)=\CC^\rel(\cM)_q.
\ee

\subsection{Relative holonomicity}\label{subsec:relhol}
We recall the following definition from \cite{BVWZ}, which is  motivated by \cite[D\'efinition 1 and Proposition 8]{Mai} and will play a key role in understanding Bernstein-Sato ideals.
\begin{definition}\cite[Definition 3.2.3]{BVWZ}
Suppose that $\cM$ is a coherent $\sA_R$-module. We say that $\cM$ is relative holonomic over $R$ (or $\Spec R$) if each irreducible component of $\Chr(\cM)$ has a decomposition as 
$\Lambda\times S,$
where $\Lambda$ is an irreducible conic Lagrangian in $T^*X$ and $S$ is an algebraic irreducible subvariety of $\Spec R$.
\end{definition}
It is worth mentioning that if $\cM$ is relative holonomic, then its graded number $j(\cM)\ge n$, by Theorem \ref{thm:relCh}(1). The category of all relative holonomic modules over $R$ is an abelian category (see \cite[3.2.4(1)]{BVWZ}).
\begin{lemma}\cite[Lemma 3.4.1]{BVWZ}
\label{lm:suppzbf}
If $\cM$ is a relative holonomic $\sA_R$-module, then 
\[Z(B_\cM)=p_2(\Chr(\cM))\textup{ and }Z(B_\cM)=\supp_{R}(\cM)\]
where $p_2\colon T^*X\times \Spec R\to \Spec R$ is the natural projection.
\end{lemma}

\subsection{Duality}\label{subsec:dualreald}
We now discuss duality of $\sA_R$-modules in a way compatible with the duality for usual $\shD$-modules. 
\begin{definition}\label{def:purecm}
Assume that $\cM$ is a non-zero coherent $($left$)$ $\sA_R$-module.

(1) The duality functor is \[\D(\cM)\coloneqq \mathscr{R}hom_{\sA_R}(\cM,\sA_R)\otimes_{\sO}\omega^{-1}_X[n],\]
where $\omega_X$ is the dualizing sheaf of $X$. 

(2) The graded number of $\cM$, denoted by $j(\cM)$, is defined by 
\[j(\cM)=j_{\sA_R}(\cM)\coloneqq \min\{k|\Ext^k_{\sA_R}(\cM, \sA_R)\neq 0\}.\]

(3) For some $j\in \Z_{\ge 0}$, $\cM$ is called $j$-pure if
\[j(\cN)=j(\cM)=j\]
for every non-zero submodule $\cN\subseteq \cM$.

(4) For some $j\in \Z_{\ge0}$, $\cM$ is called $j$-Cohen-Macaulay over $\sA_R$ if 
\[\Ext^k_{\sA_R}(\cM, \sA_R)=0,\textup{ for } k\not=j.\]
\end{definition}
One can define purity and Cohen-Macaulayness for $\gr^\rel_\bullet\sA_R$-modules in a similar way. See \cite[Appendix IV]{Bj} in general. 

Similar to the case for coherent $\shD$-modules, we have:
\begin{theorem}\cite[Theorem 3.2.2]{BVWZ}\label{thm:relCh}
Suppose that $\cM$ is a coherent $\sA_R$-module. Then \\
(1) $j(\cM)+\dim(\Chr(\cM))=2n+\dim(\Spec R)$;\\
(2) if $0\to \cM'\to \cM\to \cM''\to 0$ is a short exact sequence of coherent $\shD_X[\bs]$-modules, then 
\[\Chr(\cM)=\Chr(\cM')\cup\Chr(\cM'')\]
and if $p$ is the generic point of an irreducible component of $\Chr(\cM)$ then 
\[m_p(\cM)=m_p(\cM')+m_p(\cM'').\]
\end{theorem}

Using Part(1) of the above theorem, one can easily see that $j$-Cohen-Macaulayness implies $j$-purity. 

Since $R$ is in the center of $\sA_R$, by definition, it is obvious that localization and duality commute, that is,
\[\D(\cM)_q\simeq \D(\cM_q)\]
where $q$ is a prime ideal in $R$.

We collect the following lemma for later usage.
\begin{lemma}\label{lm:sescm}
Suppose that we have a short exact sequence of relative holonomic $\sA_R$-modules:
\[0\to \cM_1\rightarrow \cM\rightarrow\cM_2\to 0.\]
 If $\cM_1$ and $\cM$ are Cohen-Macaulay, $j(\cM_1)=j(\cM)$ and $j(\cM_2)=j(\cM)+1$, then  $\cM_2$ is Cohen-Macaulay.
\end{lemma}
\begin{proof}
We apply $\mathscr{R}hom_{\sA_R}(\bullet,\sA_R)$ to the short exact sequence. By consider the associated long exact sequence, the required statement follows immediately. 
\end{proof}

\begin{prop}\label{prop:purezerob}
Assume that $\cM$ is a relative holonomic $\sA_R$-module. If $\cM$ is $(n+k)$-pure for some integer $k\ge 0$, then $Z(B_\cM)$ is equi-dimensional of codimension $k$.
\end{prop}
\begin{proof}
By purity, we can find a good filtration $F_\bullet\cM$ over $F^\rel_\bullet\sA_R$ so that $\gr^F_\bullet\cM$ is also $(n+r)$-pure over $\gr^\rel_\bullet\sA_R$, thanks to \cite[A:IV. Theorem 4.11]{Bj}. Then 
\[\supp(\gr^F_\bullet\cM)=\Ch^\rel(\cM)\]
is equi-dimensional by \cite[A:IV. Proposition 3.7]{Bj}. Then the required statement follows from the definition of relative holonomicity and Lemma \ref{lm:suppzbf}.
\end{proof}

\subsection{Direct image functor and base change}\label{subsec:pffunctor}
Let $\mu\colon X\to Y$ be a morphism between smooth complex varieties. If $\cM$ be a left relative $\shD$-module on $X$ over $R$ (or more generally a complex of left relative $\shD$-module on $X$ over $R$), then since $\cM$ is particularly a $\shD_X$-module, we define the direct image of $\cM$ under $\mu$ by
\[\mu_+(\cM)=\R\mu_*(\cM\otimes_{\sO}\omega_X\otimes^\L_{\shD_X} \mu^*(
\shD_Y\otimes_\sO \omega_Y^{-1}))\]
where $\omega_X$ (resp. $\omega_Y$) is the dualizing sheaf of $X$ (resp. $Y$). Namely, the direct images of relative $\shD$-modules are just their direct images as absolute $\shD$-modules. It is obvious that $\mu_+(\cM)$ is a complex of left relative $\shD$-module on $Y$ over $R$. Similar to the absolute case (see \cite[Theorem 2.3.15]{Bj}), one can check that $\mu_+$ preserves coherence. 
\begin{theorem}\label{thm:pshfdualcmm}
Suppose that $\cM$ is a coherent relative $\shD$-module on $X$ over $R$, and that $\mu:X\to Y$ is a proper morphism between smooth complex varieties. Then we have a canonical isomorphism
\[\mu_+(\D\cM)\simeq \D(\mu_+\cM).\]
\end{theorem}
\begin{proof}
We first prove the case $\cM=\cN\otimes_\bbC R$ for $\cN$ a coherent $\shD_X$-module. Then 
\[\D(\cM)\simeq \D(\cN)\otimes_\bbC R\textup{ and }\mu_+\cM=(\mu_+\cN)\otimes_\bbC R.\]
Thus
\[\mu_+(\D\cM)\simeq\mu_+(\D\cN)\otimes_\bbC R.\]
By the commutativity between the direct image and the duality functor for absolute $\shD$-modules (see for instance \cite[Theorem 2.7.2]{HTT}), we have 
\[\mu_+(\D\cN)\simeq\D(\mu_+\cN).\]
Therefore,
\[\mu_+(\D\cM)\simeq\mu_+(\D\cN)\otimes_\bbC R\simeq \D(\mu_+\cN)\otimes_\bbC R\simeq \D((\mu_+\cN)\otimes_\bbC R)\simeq \D(\mu_+\cM).\]

We then can apply the method in the proof of \cite[Theorem 1.5.8]{Bj} to get a finite resolution 
\[\cP^\bullet\to \cM\]
such that each $\cP^i=\shD_X\otimes_\sO \cL^i\otimes_\bbC R$ and each $\cL^i$ is a coherent $\sO_X$-module. We then consider the distinguish triangle 
\[\tau_{\le j}(\cP^\bullet)\to \cP^\bullet \to \tau_{>j}(\cP^\bullet)\xrightarrow{+1}\]
for each $j$, where $\tau$ is the canonical truncation functor. One then can do induction on the length of $\cP^\bullet$ and reduce to the case we have proved. 
\end{proof}
Suppose that $R\to S$ is a homomorphism between  commutative rings and that 
\[\iota\colon \Spec S\rightarrow\Spec R\]
is the induced morphism of schemes. We have the derived pullback functor defined for a $\sA_R$-module $\cM$ by
\[\L\iota^*(\cM)\coloneqq \cM\otimes^\L_R S.\]
Then $\L\iota^*(\cM)$ is a complex of relative $\shD$-modules over $S$.
\begin{prop}\label{prop:relbasechange}
Suppose that $\cM$ is a coherent relative $\shD$-module over $R$, that $\mu:X\to Y$ is a proper morhism between smooth complex varieties and that $\iota:\Spec S\to \Spec R$. Then we have a natural quasi-isomorphism
\[\L\iota^*(\mu_+(\cM))\stackrel{q.i.}{\simeq} \mu_+(\L\iota^*(\cM)).\]
\end{prop}
\begin{proof}
By \cite[Proposition 2.5.13]{KSbook}, we know 
\[\L\iota^*(\mu_+(\cM))\stackrel{q.i.}{\simeq} \R\mu_*(\cM\otimes_{\sO}\omega_X\otimes^\L_{\shD_X} \mu^*(
\shD_Y\otimes_\sO \omega_Y^{-1})\otimes^\L_R S).\]
Since $R$ as well as $S$ is in the center of $\sA_R$ (resp. $\sA_S$), 
\[\R\mu_*(\cM\otimes_{\sO}\omega_X\otimes^\L_{\shD_X} \mu^*(
\shD_Y\otimes_\sO \omega_Y^{-1})\otimes^\L_R S)\simeq \mu_+(\cM\otimes_R^\L S),\]
and the proof is done. 
\end{proof}

\section{The module $\shD_X[\bs]\bff^\bs$}\label{sec:rr}
\subsection{Relative characteristic cycles and Bernstein-Sato ideal}
Suppose that $X$ is a smooth algebraic variety (or a complex manifold) of dimension $n$ and $\bff=(f_1,\dots,f_r)$ is an $r$-tuple of regular functions (or germs of holomorphic functions in the analytic case) on $X$. We write by $D$ the divisor $(\prod_i f_i=0)$ and $j\colon U=X\setminus D\hookrightarrow X$ the open embedding. We denote by
\[\bff^\bs=\prod_{i=1}^r f_i^{s_i}.\]
and write the $\sO_X$-algebra by
\[j_*(\sO_U[\bs])\coloneqq j_*\sO_U\otimes_\bbC \bbC[\bs]=j_*(\sO_U\otimes_\bbC \bbC[\bs]).\]
With the natural actions of differential operators, $j_*(\sO_U[\bs]\bff^\bs)$, the free $j_*(\sO_U[\bs])$-module generated by $\bff^\bs$, is a left $\shD_X[\bs]$-module, but not necessarily a coherent $\shD_X[\bs]$-module, that is, we assign 
\[v(\bff^\bs)=(\sum_{i=1}^r v(f_i)/f_i)\bff^\bs\]
for vector fields $v$ on $X$. 

We then consider the (left) coherent $\shD_X[\bs]$-submodule generated by $\bff^\bs$, 
$$\shD_X[\bs]\bff^\bs\subseteq j_*(\sO_U[\bs]\bff^\bs).$$
In general, for $\ba=(a_1,\dots,a_r)\in \Z^r$ we also consider the submodule generated by $\bff^{\bs+\ba}=\prod_if_i^{s_i+a_i}$, denoted by $\shD_X[\bs]\bff^{\bs+\ba}$.
\begin{remark}\label{rmk:gaga}
In the analytic case, one can replace 
$j_*\sO_U$ by 
$$\sO^\an_X(*D){=}\sO^\an_X[1/\prod_{i=1}^rf_i],$$ the algebraic localization of $\sO^\an_X$ along $D$, and construct $\shD_{X}^{\an}[\bs]\bff^{\bs+\ba}$ similarly. 
\end{remark}


If $\ba\ge \bb \in \Z^r$ (that is, $a_i\ge b_i$ for every $i\in \{1,2,\cdots,r\}$), we obviously have 
\[\shD_X[\bs]\bff^{\bs-\bb} \subseteq \shD_X[\bs]\bff^{\bs-\ba}.\]
We then write the quotient module by 
\[\cM^{\ba,\bb}_\bff\coloneqq \dfrac{\shD_X[\bs]\bff^{\bs-\ba}}{\shD_X[\bs]\bff^{\bs-\bb}}.\]
As an $\sO_X$-module, it is supported on $D_{\ba - \bb}$, where $D_\ba=(\prod_{a_i\not=0}f_i=0)$.

The modules $\shD_X[\bs]\bff^{\bs-\ba}$ and $\cM^{\ba,\bb}_\bff$ are $\shD$-modules over $\bbC[\bs]$.
We denote the $\bbC[\bs]$-module annihilator of $\cM^{\ba,\bb}_\bff$ by
\[B^{\ba,\bb}_\bff\coloneqq B_{\cM^{\ba,\bb}_\bff}= \Ann_{\bbC[\bs]}(\cM^{\ba,\bb}_\bff )\subseteq \bbC[\bs],\]
called the Bernstein-Sato ideals of $\bff$ with indices $\ba\ge\bb$.

For simplicity, we write $\cM^{\mathbf{0},\bb}_\bff$ by $\cM^{\bb}_\bff$ and $B^{\mathbf 0, \bb}_\bff$ by $B^{\bb}_\bff$ when $\mathbf 0\ge \bb$, and $B^{-\mathbf 1}_\bff$ by $B_{\bff}$, where $\mathbf 1$ denotes the vector $(\underbrace{1,1,\dots,1}_r)$. We denote by $Z(B_\bff^{\ba,\bb})$ the zero locus of $B^{\ba,\bb}_\bff$ and by $Z_l(B_\bff^{\ba,\bb})$ the part of $Z(B_\bff^{\ba,\bb})$ with pure dimension $l$.

When $r=1$, $B_f$ is a principle ideal in $\bbC[s]$, we write the monic polynomial generating $B_f$ by $b_f(s)$; it is the usual Bernstein-Sato polynomial (or $b$-function) for $f$ (see \cite{KasBf}).    



The following theorem is essentially due to Maisonobe with Part(2) improved in \cite{BVWZ2}:
\begin{theorem}\label{thm:maisonmain}
We have \\
(1) $\shD_X[\bs]\bff^{\bs+\ba}$ is relative holonomic and $n$-pure for every $\ba\in \Z^r$;\\
(2) $\CC^\rel(\shD_X[\bs]\bff^{\bs+\ba})=\CC(j_*\sO_U)\times\bbC^r$;\\
(3) $\dim(\Chr(\cM^{\ba,\bb}_\bff))=n+r-1$ for every pair $\ba>\bb\in \Z^r$ and hence $$j(\cM^{\ba,\bb}_\bff)=n+1;$$\\
(4) $\cM^{\ba,\bb}_\bff$ is relative holonomic for every pair $\ba>\bb\in \Z^r$;\\
(5) $p_2(\Chr(\cM^{\ba,\bb}_\bff))=Z(B^{\ba,\bb}_\bff)$ for every pair $\ba>\bb\in \Z^r$, where 
$$p_2\colon T^*X\times \bbC^r\to \bbC^r$$
is the projection.
\end{theorem}
\begin{proof}
Part (1) and Part (2) for the relative characteristic varieties follow from \cite[R\'esultat 1 and Propostion 14]{Mai}. Part (3) follows from  R\'esultat 2 in \emph{loc. cit.} and Proposition \ref{prop:CCdec} below, while Part (4) follows from Lemma \ref{lm:suppzbf}.  See \cite[Theorem 4.3.4]{BVWZ2} for Part (2) in cycles. 
\end{proof}


\begin{prop}\label{prop:CCdec}
For $\ba,\bb,\bc\in \Z^r$, if $\ba\ge\bb\ge\bc$, then
\[\CC^\rel_{n+r-1}(\cM^{\ba,\bc}_\bff)=\CC^\rel_{n+r-1}(\cM^{\ba,\bb}_\bff)+\CC^\rel_{n+r-1}(\cM^{\bb,\bc}_\bff),\]
and 
\[Z_{r-1}(B^{\ba,\bc}_\bff)=Z_{r-1}(B^{\ba,\bb}_\bff)\bigcup Z_{r-1}(B^{\bb,\bc}_\bff).\]
\end{prop}
\begin{proof}
Since $\ba\ge\bb\ge\bc$, we have a short exact sequence
\[0\to \cM^{\bb,\bc}_\bff\rarr \cM^{\ba,\bc}_\bff\rarr \cM^{\ba,\bb}_\bff\to 0\]
Since $\dim(\Chr(\cM^{\ba,\bb}_\bff))=n+r-1$, the first statement follows from Theorem \ref{thm:relCh}(2). The second then follows from Theorem \ref{thm:maisonmain}(5).
\end{proof}

\begin{theorem}\label{thm:sabgj}
(1) \cite{Sab2} There exists a polynomial
\[b(\bs)=\prod_{L,\alpha}(L\cdot\bs+\alpha)\in B_\bff\]
where the product is over finite many (possibly repeated) $L\in \Z^r_{\ge 0}$ and $\alpha\in \Q$.\\
(2) \cite{Gyo} There exists a polynomial
\[b(\bs)=\prod_{L,\alpha\in \Q_{>0}}(L\cdot\bs+\alpha)\in B_\bff,\]
where the product is over finite many $L\in \Z^r_{\ge 0}$ and $\alpha\in \Q_{>0}$.
\end{theorem}

By the above theorem and Theorem \ref{thm:maisonmain} (3), (4) and (5) and Proposition \ref{prop:purezerob}, we immediately have:
\begin{coro}\label{cor:codime1}
There exist decompositions
\[\Ch^\rel_{n+r-1}(\cM^{-\mathbf 1}_\bff)=\sum_{L\ge{\mathbf 0},\alpha>0} \Lambda_{L,\alpha}\times (L\cdot\bs+\alpha=0)\]
and 
\[\Ch^\rel_{n+r-1}(\cM^{\bk,-\mathbf \bk}_\bff)=\sum_{L\ge{\mathbf 0},\alpha\in \Q} \Lambda_{L,\alpha}\times (L\cdot\bs+\alpha=0)\]
for all pairs $\ba\ge \bb$
and $\Lambda_{L,\alpha}$ are conic Lagrangian subvarieties in $T^*X$, and 
$$Z_{r-1}(B_\bff)=\bigcup_{L\ge{\mathbf 0},\alpha>0} (L\cdot\bs+\alpha=0)$$
and 
$$Z_{r-1}(B^{\ba,-\bb}_\bff)=\bigcup_{L\ge{\mathbf 0},\alpha\in \Q} (L\cdot\bs+\alpha=0).$$
\end{coro}

For $i\in \{1,2,...,r\}$, we write by $\bee_i$ the unit vector with 1 at the position $i$. Corollary \ref{cor:codime1} and Proposition \ref{prop:CCdec} says that:
\begin{coro}\label{cor:slopei}
$Z_{r-1}(B^{-\bee_i}_\bff)$ is a finite union of divisors $(L\cdot\bs+\alpha=0)$ with $L>{\mathbf 0}$ and $\alpha>0\in\Q$.
\end{coro}

For higher codimensional parts, Maisonobe gave the following estimate (see \cite[R\'esultat 3]{Mai}):
\begin{theorem}[Maisonobe]
Every irreducible component of $Z(B_\bff)$ of codimension $>1$ can be translated by an element of $\Z^r$ inside a component of $Z_{r-1}(B_\bff)$.
\end{theorem}

When $\cM_\bff^{-\mathbf1}$ is pure, by Proposition \ref{prop:purezerob} one can improve the result in the above theorem as follows.
\begin{prop}\label{prop:ZbfpureMf}
If the $\cM^{-\mathbf 1}_\bff$ is $(n+1)$-pure, then 
\[Z(B_\bff)=Z_{r-1}(B_\bff).\]
\end{prop}

\subsection{Set of slopes and generalization of the log canonical threshold}
We define the set of slopes
\[S_\bff\coloneqq\{L \textup{ primitive vectors in } \Z^r_{\ge0} \mid (L\cdot\bs+\alpha=0 )\subseteq Z_{r-1}(B_\bff) \textup{ for some } \alpha\in\Q\}\]
and define for $i\in \{1,2,...,r\}$
\[
S_{\bff,i}\coloneqq\{L\textup{ primitive vectors in } \Z^r_{\ge0}\mid (L\cdot\bs+\alpha=0)\subseteq Z_{r-1}(B^{-\bee_i}_\bff) \textup{ for some } \alpha\in\Q\}.\]
Since $L\in S_\bff$ are primitive in the lattice $\Z^r$, $S_\bff$ and $S_{\bff,i}$ are finite sets.  
By Proposition \ref{prop:CCdec}, we have 
\be\label{eq:uslopes}
S_\bff=\bigcup_{i=1}^rS_{\bff,i}.
\ee
For $i\in \{1,2,...,r\}$ and $L>0\in \Z^r$, we define 
\[ \kappa(L,i)=\kappa_\bff(L,i)\coloneqq
\begin{cases}
 \min\{\alpha\mid (L\cdot \bs+\alpha=0)\subseteq Z_{r-1}(B^{-\bee_i}_\bff)\}, \textup{ if } L\in S_{\bff,i}\\
\infty, \textup{ if } L\notin S_{\bff,i}
\end{cases}
\]
and
\[\kappa(L)=\kappa_\bff(L)\coloneqq\min\{\kappa(L,i)\mid i=1,2,\dots,r\}.\]
By Eq.\eqref{eq:uslopes} and Corollary \ref{cor:slopei}, if $L\in S_\bff$, then for some $i$
\be\label{eq:moffset}
0<\kappa(L)=\kappa(L,i)<\infty.
\ee
\begin{remark}
$\kappa(\cL)$ is the generalization of the log canonical threshold of the divisor $(f=0)$ when $r=1$. Indeed, when $r=1$, $S_f=\{1\}$ and $\kappa(1)$ is the log canonical threshold of the divisor $(f=0)$ (see for instance \cite{BSaito} and \cite{ELSVD}).
\end{remark}

\begin{lemma}\label{lm:codoneloc}
Suppose that $(L\cdot\bs+\alpha=0)$ is a divisor in $\bbC^r$ with $L>0\in \Z^r$ and $\alpha\in \Q$ and $q$ is the prime ideal generated by $L\cdot\bs+\alpha$. If  
\[(L\cdot\bs+\alpha=0)\subseteq Z_{r-1}(B_\bff^{\ba,\ba-e_i})\]
then $B_{\bff,q}^{\ba,\ba-e_i}=((L\cdot\bs+\alpha)^{n_{L,\alpha}})\subseteq \bbC[\bs]_q$
for some integer $n_{L,\alpha}>0$. Moreover, if 
\[(L\cdot\bs+\alpha=0)\not\subseteq Z_{r-1}(B_\bff^{\ba,\ba-e_i})\]
then $B_{\bff,q}^{\ba,\ba-\bee_i}=\bbC[\bs]_q$.
\end{lemma}
\begin{proof}
The second case is obvious by definition. In the first case, by definition we know $B_{\bff,q}^{\ba,\ba-\bee_i}\not=\bbC[\bs]_q.$ 
Since $\bbC[\bs]_q$ is a DVR and hence a PID, the required statement then follows.  
\end{proof}


\begin{prop}\label{prop:nonjump}
Let $L\in S_\bff$, $\alpha \in \Z$, and  $q$ be the prime ideal in $\C[\bs]$ generated by $L\cdot\bs+\alpha$. Then for every $\ba \in \Z^r$ such that $L \cdot  \ba < \kappa(L)-\alpha$, the modules
\[\shD_X[\bs]_q\bff^{\bs-\ba}\]
are all the same. 
\end{prop}
\begin{proof}
It suffices to prove that, for all $\ba$ satisfying $L \cdot  \ba < \kappa(L)-\alpha$, and for every $i\in \{1,\cdots,r\}$, we have 
\[\shD_X[\bs]_q\bff^{\bs-\ba} = \shD_X[\bs]_q\bff^{\bs-(\ba-\bee_i)}. \]
Indeed, if so, then all lattice points with $L \cdot  \ba < \kappa(L)-\alpha$ are connected by these equalities. 

The $\supseteq$ direction is clear, we now prove $\subseteq$.  Note that \[ Z_{r-1}(B_\bff^{\ba,\ba-\bee_i}) = Z_{r-1}(B_\bff^{-\bee_i}) + \ba, \] 
since one is replacing $\bs$ by $\bs-\ba$ in all the equations. Thus, the ``first" hyperplane in the $L$ direction that one encounters in $Z_{r-1}(B_\bff^{\ba,\ba-\bee_i})$ becomes $(L \cdot (\bs - \ba) + \kappa(L,i) =0)$. Since 
\[ \alpha < \kappa(L) - L \cdot \ba \leq \kappa(L,i) - L \cdot \ba, \]
we have \[(L\cdot\bs+\alpha=0)\not\subseteq Z_{r-1}(B_\bff^{\ba,\ba-\bee_i}).\]
By Lemma \ref{lm:codoneloc}, we obtain that $\cM^{\ba,\ba-\bee_i}_{\bff,q}=0$ and hence 
$$\shD_X[\bs]_q\bff^{\bs-\ba+\bee_i} = \shD_X[\bs]_q\bff^{\bs-\ba}.$$
\end{proof}

\subsection{The maximal extension and the $j_!$ extension}\label{subsec:j_*j_!}
The module $j_*(\sO_U[\bs]\bff^\bs)$, called the maximal extension, is the ambient module where all $\shD_X[\bs]\bff^{\bs+\ba}$ live inside. Globally, it is not coherent over $\shD_X[\bs]$. However, after localization we have:  
\begin{theorem}\label{thm:j_*loc}
Let $q$ be a prime ideal in $\bbC[\bs]$.
Then $j_*(\sO_U[\bs]\bff^{\bs})_q$ is $n$-Cohen-Macaulay and 
\[j_*(\sO_U[\bs]\bff^{\bs})_{q}=\shD_X[\bs]_q\bff^{\bs-\bk}\]
for all $k\gg 0$, where $\bk=(k,k,\dots,k)$.
\end{theorem}
\begin{proof}
For a maximal ideal $m$, the equality 
\[j_*(\sO_U[\bs]\bff^{\bs})_{m}=\shD_X[\bs]_{m}\bff^{\bs-\bk}\]
is a special case of \cite[Theorem 5.3(ii)]{WZ}. The $n$-Cohen-Macaulayness is contained in its proof in \emph{loc. cit.} See also \cite[\S 5]{BVWZ2}.

In general, we take a maximal ideal so that $q\subseteq m$. Since duality and localization commute (cf. \S \ref{subsec:dualreald}), the proof is done. 
\end{proof}
For a prime ideal $q\subset \bbC[\bs]$, we now define 
\[j_!(\sO_U[\bs]\bff^\bs)_q\coloneqq \D_X (j_*(\D_U((\sO_U[\bs]\bff^\bs)_q).\]
One can easily see that $\sO_U[\bs]_{q}\bff^\bs$ 
is $n\textup{-Cohen-Macaulay}$ 
(cf. \cite[Lemma 5.3.1]{BVWZ2}). Therefore, by Theorem \ref{thm:j_*loc}, $j_!(\sO_U[\bs])_{q}$ is a sheaf (instead of a complex) and it is $n$-Cohen-Macaulay. Since $\D\circ\D$ is identity, using the adjunction pair $(j^{-1},j_*)$, we obtain a natural morphism 
\[j_!(\sO_U[\bs]\bff^\bs)_q\rightarrow j_*(\sO_U[\bs]\bff^\bs)_q.\]
\begin{theorem}\label{thm:j_!loc}

Let $q$ be a prime ideal in $\bbC[\bs]$.
Then $j_!(\sO_U[\bs]\bff^\bs)_{m_\bal}$ is $n$-Cohen-Macaulay, the natural morphism 
\[j_!(\sO_U[\bs]\bff^\bs)_q\hookrightarrow j_*(\sO_U[\bs]\bff^\bs)_{q}\]
is injective and 
\[j_!(\sO_U[\bs]\bff^\bs)_{q}=\shD_X[\bs]_{q}\bff^{\bs+\bk}\]
for all $k\gg 0$, where $\bk=(k,k,\dots,k)$. 
\end{theorem}
\begin{proof}
For a maximal ideal $m$, the injectivity and the identity
\[j_!(\sO_U[\bs]\bff^\bs)_{m}=\shD_X[\bs]_{m}\bff^{\bs+\bk}\]
for all $k\gg 0$ follows from \cite[Theorem 5.4(iii)]{WZ}. One then takes a maximal ideal $m$ such that $q\subseteq m$ and takes further localization.
\end{proof}

\section{Diagonal Specializations}\label{sec:spe}
We continue using notations introduced in \S \ref{sec:rr} and suppose that $\bff=(f_1,\dots,f_r)$ is a $r$-tuple of regular functions (or holomorphic function) on a smooth complex variety ( or a complex manifold) $X$ for some $r\ge 2$.

Let us consider the following diagonal embedding:
\[ \Delta: \C^{r-1} \to \C^r, \bw \mapsto \bs, \quad s_j = w_j \textup{ for } j\le r-1 \textup{ and } s_r=w_{r-1}, \]
where $\bw=(w_1,\dots,w_{r-1})$ is the algebraic coordinates of the affine space $\bbC^{r-1}$.
We write by 
\[\bg=(f_1,\dots,f_{r-2}, f_{r-1}f_{r})\textup{ and }\bg^\bw=\prod_{i=1}^{r-1} g_i^{w_i}.\]

We assume that $\cM$ is a relative $\shD$-module over $\bbC^r$.  Since it is particularly a $\bbC[\bs]$-module, we consider the pullback complex 
$$\L\Delta^*\cM\coloneqq\cM\otimes^L_{\bbC[\bs]}\C[\bw]$$
is a complex of relative $\shD$-module over $\bbC^{r-1}$.
We also write 
\[\Delta^*\cM=\L^0\Delta^*\cM\simeq\cM\otimes_{\bbC[\bs]}\C[\bw].\]

\begin{lemma}\label{lm:pbj_*}
With notations as above, we have 
$$\Delta^*(j_*(\sO_U[\bs]_{}\bff^\bs))\simeq j_*(\sO_U[\bw]\bg^\bw)$$ 
as $\shD_X[\bw]$-modules.
\end{lemma}
\begin{proof}
Since $\sO_U[\bs]\bff^\bs$ is a free $\sO_U[\bs]$-module generated by $\bff^\bs$ and the functor $j_*$ is exact, $\Delta^*(j_*(\sO_U[\bs]\bff^\bs))$ is a free $j_*(\sO_U[\bw])$-module of rank 1 and hence  
\[\Delta^*(j_*(\sO_U[\bs]\bff^\bs))\simeq j_*(\sO_U[\bw]\bg^\bw).\]
It is obvious that the $P$-actions on both $\Delta^*(j_*(\sO_U[\bs]\bff^\bs)$ and $j_*(\sO_U[\bw]\Delta^*\bff^\bs)$ are compatible with the above isomorphism for differential operators $P\in \shD_X$, and hence the above isomorphism of $j_*(\sO_U[\bw])$-modules is indeed an isomorphism of $\shD_X[\bw]$-modules. 

\end{proof}

\begin{lemma}\label{lm:sujpb}
With notations in Lemma \ref{lm:pbj_*}, we have an natural surjective morphisms 
\[\Delta^*(\shD_X[\bs]\bff^{\bs+\ba})\to \shD_X[\bw](\bff^\ba\cdot\bg^\bw).\]
\end{lemma}
\begin{proof}
Pulling back the inclusion
\[\shD_X[\bs]\bff^{\bs+\ba)}\hookrightarrow j_*(\sO_U[\bs]\bff^\bs),\]
by Lemma \ref{lm:pbj_*} we get a natural morphism of $\shD_X[\bw]$-modules,
\[\Delta^*(\shD_X[\bs]\bff^{\bs+\ba})\rightarrow j_*(\sO_U[\bw]\bg^\bw)\simeq\Delta^*(j_*(\sO_U[\bs]\bff^\bs))),\]
whose image is obviously generated by $\bff^\ba\cdot\bg^\bw$. 

\end{proof}

\begin{theorem}\label{thm:diagpbcm}
With notations in Lemma \ref{lm:pbj_*}, we assume that $\shD_X[\bs]\bff^\bs$ is $n$-Cohen-Macaulay over $\shD_X[\bs]$. Then 
$\shD_X[\bw]\bg^\bw$ is $n$-Cohen-Macaulay over $\shD_X[\bw]$ and
\[Z(B_{\bg})=\Delta^{-1}(Z(B_{\bff})).\]
\end{theorem}
\begin{proof}
By construction of $\Delta$, $\Delta(\bbC^{r-1})$ is the smooth divisor $$(s_{r-1}-s_r=0).$$ 
Then for $\bbC[\bs]$-module $\cM$, we have
\[\Delta^* \cM\simeq\cM\otimes_{\bbC[\bs]}\frac{\bbC[\bs]}{(s_{r-1}-s_r)}.\]
For $k\ge 0$ and $\bk\coloneqq(\underbrace{k,k,\dots,k}_r)$, we consider the following commutative diagram
\[
\begin{tikzcd}
   &0\arrow[d] & 0\arrow[d] & 0\arrow[d]& \\
  0 \arrow[r]& \shD_X[\bs]\bff^{\bs-\bk+\mathbf 1}\arrow[r]\arrow[d, "\cdot (s_{r-1}-s_r)"] &\shD_X[\bs]\bff^{\bs-\bk}\arrow[r] \arrow[d, "\cdot (s_{r-1}-s_r)"]&\cM^{\bk,\bk-\mathbf{1}}_{\bff} \arrow[r] \arrow[d, "\cdot (s_{r-1}-s_r)"]& 0  \\
    0 \arrow[r]& \shD_X[\bs]\bff^{\bs-\bk+\mathbf 1}\arrow[r]\arrow[d] &\shD_X[\bs]\bff^{\bs-\bk}\arrow[r]\arrow[d] &\cM^{\bk,\bk-\mathbf{1}}_{\bff} \arrow[r] \arrow[d]& 0  \\
 0 \arrow[r]& \Delta^*(\shD_X[\bs]\bff^{\bs-\bk+\mathbf 1})\arrow[r] \arrow[d]&\Delta^*(\shD_X[\bs]\bff^{\bs-\bk})\arrow[r]\arrow[d] &\Delta^*(\cM^{\bk,\bk-\mathbf{1}}_{\bff}) \arrow[r] \arrow[d]& 0.  \\
        & 0&0&0&   
\end{tikzcd}
\]
Since $\shD_X[\bs]\bff^{\bs+\bk}\subseteq j_*(\sO_U[\bs])\bff^\bs$ and the later is free over $\bbC[\bs]$, $\shD_X[\bs]\bff^{\bs+\bk}$ is torsion free for every $\bk=(k,k,\dots,k)\in\Z^r$ and hence the first two columns are exact.  Since  $\shD_X[\bs]\bff^{\bs}$ is $n$-Cohen-Macaulay, so are $\shD_X[\bs]\bff^{\bs+\bk}$ by substitution. Hence, by Lemma \ref{lm:sescm}, $\cM_\bff^{\bk,\bk-\mathbf1}$ are $(n+1)$-Cohen-Macaulay. We then can apply \cite[Lemma 3.4.2]{BVWZ} and conclude that the third column is also exact. Therefore, by $3\times 3$ Lemma, the third row is also exact. 
We hence obtained a direct system of inclusions
\[\{\Delta^*(\shD_X[\bs]\bff^{\bs-\bk+\mathbf 1})\hookrightarrow\Delta^*(\shD_X[\bs]\bff^{\bs-\bk})\}_{k\ge 0}.\]
Since 
$$\lim_{k\to\infty}\shD_X[\bs]\bff^{\bs-\bk}=j_*(\sO_U[\bs])\bff^\bs,$$
using the fact that direct limit functor is exact and commute with tensor-product we have inclusions 
\[\Delta^*(\shD_X[\bs]\bff^{\bs+\mathbf 1_r})\hookrightarrow \Delta^*(\shD_X[\bs]\bff^{\bs})\hookrightarrow \Delta^*(j_*(\sO_U[\bs])\bff^\bs).\]
By Lemma \ref{lm:pbj_*} and Lemma \ref{lm:sujpb}, we thus obtain that 
$$\Delta^*(\shD_X[\bs]\bff^{\bs+\mathbf 1_r})\simeq\shD_X[\bw]\bg^{\bw+\mathbf{1}_{r-1}}\textup{ and }\Delta^*(\shD_X[\bs]\bff^{\bs})\simeq\shD_X[\bw]\bg^\bw,$$
where $\mathbf 1_r=(1,1,\dots,1)\in\Z^r$. Hence, 
\[\Delta^*(\cM_\bff^{-\mathbf 1_r})\simeq \cM_\bg^{-\mathbf 1_{r-1}}.\]
Since $\Delta^*(\shD_X[\bs]\bff^{\bs})$ is annihilated by $(s_{r-1}-s_r)$, we can apply the Rees theorem in homological algebra (see for instance \cite[Theorem 8.34]{Rotm}) and conclude that 
\[\Ext^l_{\shD_X[\bw]}(\Delta^*(\shD_X[\bs]\bff^{\bs}),\shD_X[\bw])\simeq\Ext^{l+1}_{\shD_X[\bs]}(\Delta^*(\shD_X[\bs]\bff^{\bs}),\shD_X[\bs]).\]
We consider the first column (for $k=1$) of the above $3\times3$ diagram and conclude that $\Delta^*(\shD_X[\bs]\bff^{\bs})$ is $(n+1)$-Cohen-Macaulay over $\shD_X[\bs]$ by Lemma \ref{lm:sescm}(1). Thus, $\Delta^*(\shD_X[\bs]\bff^{\bs})$ and hence $ \shD_X[\bw]\bg^{\bw}$
are $n$-Cohen-Macaulay over $\shD_X[\bw]$.

We now consider the third column of the $3\times 3$-diagram above for $k=0$. Since $\cM^{-\mathbf 1_r}_\bff$ is $(n+1)$-Cohen-Macaulay (and hence $(n+1)$-pure), we can apply Proposition \ref{prop:ZbfpureMf}. We thus conclude that $s_{r-1}-s_r$ is not contained in any minimal prime ideal containing $B_{\bff}$ by Corollary \ref{cor:codime1}. Then we consider the morphism by multiplication,
\[\cM^{-\mathbf 1_r}_\bff\xrightarrow{\cdot (s_{r-1}-s_r)}\cM^{-\mathbf 1_r}_\bff.\]
Thanks to \cite[Lemma 3.4.2]{BVWZ} again, we have a relative good filtration $F_\bullet(\cM^{-\mathbf 1_r}_\bff)$ over $F^\rel_\bullet\shD_X[\bs]$ so that 
\[\gr^F_\bullet(\cM^{-\mathbf 1_r}_\bff)\xrightarrow{\cdot (s_{r-1}-s_r)} \gr^F_\bullet(\cM^{-\mathbf 1_r}_\bff)\]
is also injective. Therefore, 
\be\label{eq:inductivechr}
\Ch^\rel(\cM_\bg^{-\mathbf1_{r-1}})=\Ch^\rel(\Delta^*(\cM^{-\mathbf1}))=\Ch^\rel(\cM^{-\mathbf1_r}))|_{s_{r-1}=s_r}.
\ee
By Lemma \ref{lm:suppzbf}, we thus have
\[Z(B_{\bg})=\Delta^{-1}(Z(B_{\bff})).\]
\end{proof}
Recently, Bath \cite[Proposition 2.29]{Bath} studied a similar specialization problem but with the Cohen-Macaulay hypothesis replaced by some geometric requirement (but more restrictive). 
\begin{proof}[Proof of Theorem \ref{thm:maindiasp}]
We inductively apply Theorem \ref{thm:diagpbcm} until $r=1$ and Theorem \ref{thm:maindiasp} follows.
\end{proof}

\section{Hyperplane arrangements}\label{sec:apptohyparr}
\subsection{A Cohen-Macaulay criterion}
We recall a Cohen-Macaulay criterion of $\shD_X[\bs]\bff^\bs$ with the help of using free divisors (in the sense of K. Saito \cite{KSaito}). Since freeness is an analytic notion, we suppose that $\bff=(f_1,\dots,f_r)$ is a $r$-tuple of holomorphic functions on $X=\bbC^n$ for some $r\ge 1$.  We keep the notation from \S \ref{sec:rr} but under the analytic setting (see Remark \ref{rmk:gaga}).

\begin{definition}\label{def:freediv}
Suppose that $h$ is a holomorphic function on $X$ and $D_h$ is the divisor $(h=0)$. Denote by $\sI_{D_h}$ the ideal sheaf of $D_h$. We let $\shTA_X(-\log D_h)$ be the sheaf of holomorphic logarithmic vector fields along $D_h$, that is , the sheaf is generated by vector fields $v$ so that $v\cdot \sI_{D_h}\subseteq \sI_{D_h}$. Then $h$, as well as $D_h$, is called free if $\shTA_X(-\log D_h)$ is a locally free $\sO_X$-module.
\end{definition}
The following theorem is first observed by Narva\'ez-Macarro \cite{NM15} when $r=1$; Maisonobe \cite{Maihyp} generalized it 
in general. 
\begin{theorem}[Maisonobe]\label{thm:freecm}
suppose that $\bff=(f_1,\dots,f_r)$ is an $r$-tuple of holomorphic functions on $X=\bbC^n$. 
If $\prod_{i=1}^r f_i$ is locally quasi-homogeneous and free, then $\shD_X[\bs]\bff^\bs$ is $n$-Cohen-Macaulay. 
\end{theorem}
\begin{proof}
This lemma indeed holds unconditionally when $r=1$. We first prove this case. By Theorem \ref{thm:maisonmain}(1), $\shD_X[s]\bff^\bs$ is $n$-pure. When $r=1$, by for instance \cite[Lemma 4.3.3(2)]{BVWZ} and Auslander regularity, we know that $$j(\mathscr{E}xt^{n+j}_{\shD_X[s]}(\shD_X[s]f^s,\shD_X[s]))>n+j.$$ If $\mathscr{E}xt^{n+j}_{\shD_X[s]}(\shD_X[s]f^s,\shD_X[s])\not= 0$ for $j>0$, then by Theorem \ref{thm:relCh}(1), $$\dim(\Ch^\rel(\mathscr{E}xt^{n+j}_{\shD_X[s]}(\shD_X[s]f^s,\shD_X[s])))<n-j+1.$$
But by \cite[Lemma 3.2.4(2)]{BVWZ}, $\mathscr{E}xt^{n+j}_{\shD_X[s]}(\shD_X[s]f^s,\shD_X[s])$ is relative holonomic over $\bbC[s]$ and hence 
\[\dim(\Ch^\rel(\mathscr{E}xt^{n+j}_{\shD_X[s]}(\shD_X[s]f^s,\shD_X[s])))\ge n.\]
This is a contradiction and hence $\mathscr{E}xt^{n+j}_{\shD_X[s]}(\shD_X[s]f^s,\shD_X[s])=0$ for $j\ge 0$, which implies that $\shD_X[s]f^s$ is $n$-Cohen-Macaulay.  

So the real content of this lemma is for the case $r\ge 2$. When $r\ge 2$, it is implied immediately by  
\cite[Proposition 5]{Maihyp}. Indeed, the Spencer complex in \cite[Proposition 5]{Maihyp} is a length-$n$ free resolution of $\shD_X[\bs]\bff^\bs$ as $\shD_X[\bs]$-modules. Hence,
\[\mathscr{E}xt^{n+j}_{\shD_X[\bs]}(\shD_X[\bs]\bff^\bs, \shD_X[\bs])=0, \textup{ for } j>0.\]
Since $j(\shD_X[\bs]\bff^\bs)=n$,  $\shD_X[\bs]\bff^\bs$ is $n$-Cohen-Macaulay. 
\end{proof}

\subsection{Proof of Theorem \ref{thm:mainrootoffha}}
Maisonobe obtained the following formula to compute the Bernstein-Sato ideals for free hyperplane arrangements. We write by $\bff_D=(f_1,f_2,\dots)$ a complete factorization of $f_D$ for a hyperplane arrangement $D$.
\begin{theorem}\cite[Th\'eor\`eme 1]{Maihyp}\label{thm:mainBffha}
Suppose that $D$ is a free hyperplane arrangement in $\bbC^n$. Then the Bernstein-Sato ideal $B_{\bff_D}$ is principal and generated by 
\[\prod_{\textup{dense} W\in L(D)}\prod_{j=0}^{2(|J(W)|-\textup{rank}(W))}(\sum_{i\in J(W)}s_i+\textup{rank}(W)+j).\]
\end{theorem}

Combining Theorem \ref{thm:maindiasp}, Theorem \ref{thm:freecm} and Theorem \ref{thm:mainBffha} together, we conclude Theorem \ref{thm:mainrootoffha}.

\subsection{Zero loci of Bernstein-Sato ideals for hyperplane arrangements}
The following lemma is a special case of \cite[Lemma 6.4]{Budur}.
\begin{lemma}\label{lm:mchyp}
Let $\bff=(f_1,\dots,f_r)$ be a complete factorization of an irreducible, essential, central hyperplane arrangement $f$. Then 
\[\sum_{i=1}^r s_i+k=0\]
defines an irreducible component of $Z(B_{\bff})$ for some $k\in \Z_{>0}$.
\end{lemma}



We denote for a factorization $\bff=(f_1,\dots,f_r)$ of a hyperplane arrangment $f$ and for $\ba\in \Z^r$ 
\[\cN_\ba\coloneqq \dfrac{\shD_X[\bs]\bff^{\bs-\ba}}{\sum_{i=1}^r\shD_X[\bs]\bff^{\bs-\ba+\bee_i} }.\]
\begin{lemma}\label{lm:sumb}
Let $\bff=(f_1,\dots,f_r)$ be a complete factorization of a central hyperplane arrangment $f$. If a subset of all irreducible components of $f$ gives a coordinate system of $X=\bbC^n$, then  
$$\sum_{i=1}^r(s_i-a_i)+n\in B_{\cN_\ba}.$$
\end{lemma} 
\begin{proof}
Without loss of generality we can assume that $(f_1,f_2,\dots,f_n)$ give a coordinate system of $\bbC^n$, which we rename to $w_i$ to avoid confusion.
One first observes that 
\[0 = (\sum_{i=1}^r s_i - \sum_{j=1}^n w_i \pa_{w_i})\bff^{\bs}= (\sum_{i=1}^r s_i + n - \sum_{j=1}^n  \pa_{w_i} w_i)\bff^{\bs} \]
Hence
\[
  (\sum_{i=1}^r s_i + n) \bff^\bs = \sum_{i=1}^n \partial_{w_i}\bff^{\bs+\bee_i}
\]
Thus, 
\[\sum_{i=1}^rs_i+n\in B_{\cN_{\mathbf 0}}.\]
For the general case, one substitutes $s_i$ by $s_i-a_i$.

\end{proof}

\begin{theorem}\label{thm:main1}
Let $\bff=(f_1,\dots,f_r)$ be a complete factorization of an irreducible, essential, central hyperplane arrangement $f$. Then $\kappa(\mathbf 1)=\kappa(\mathbf1,i)=n$, where 
$\mathbf 1=(1,1,\dots,1)\in\Z^r$. In particular,
$$\sum_{i=1}^r s_i+n=0$$
defines a component of $Z_{r-1}(B_\bff)$.
\end{theorem}
\begin{proof}
We first prove that $\kappa({\mathbf 1})=n$. By Lemma \ref{lm:mchyp}, we know $0<\kappa(\mathbf 1)<\infty$.
Suppose $\kappa(\mathbf1) \neq n$ and let $q$ be the prime ideal generated by $\mathbf 1\cdot \bs+\kappa(\mathbf1)$. Then we have $ \cM_{\bff,q}^{-\mathbf1}\not=0.$ Since $\mathbf 1\cdot \bs+n\not\in q$, and it annihilates $\cN_{\bf 0}$ by Lemma \ref{lm:sumb}, hence $\cN_{\mathbf 0,q}=0.$ 
However, thanks to Proposition \ref{prop:nonjump}, taking $\alpha = \kappa(\mathbf1)$, we have
\[\shD_X[\bs]_{q}\bff^{\bs+\bee_i}=\shD_X[\bs]_{q}\bff^{\bs+\mathbf 1}, \quad  \forall i=1,\cdots,r.\]
Hence $\cN_{\mathbf 0,q}=\cM^{-\mathbf 1}_{\bff,q}.$
which is a contradiction. Hence $\kappa({\mathbf 1})=n$. 
$$(s_1+\dots+s_r+n=0)$$
is a component of $Z_{r-1}(B_\bff)$.

By symmetry and Proposition \ref{prop:CCdec}, one can easily see $\kappa(\mathbf1)=\kappa(\mathbf1,i)$ for all $i$.
\end{proof}
\begin{proof}[Proof of Theorem \ref{thm:maingndcf}]
We first deal with the case that $f$ is central, essential and irreducible. For simplicity we write $\mathbf1_l=(\underbrace{1,1,\dots,1}_l,0,0,\dots,0)\in \Z^r.$ For $0<l\le r$,  we consider $\cM^{-\mathbf1_{l-1},-\mathbf1_l}_\bff$, which is a subquotient of $\cM_\bff$.  By Theorem \ref{thm:main1}, we know 
\[(\sum_{j=1}^rs_j+n=0)\subset Z(B_\bff^{\bee_i}).\]
By substitution, we hence know for $0<l\le r$
\[(\sum_{j=1}^rs_j+n+l-1=0)\subset Z(B_\bff^{-\mathbf1_{l-1},-\mathbf1_l}).\]
By Proposition \ref{prop:CCdec}, we hence know 
\[(\sum_{j=1}^rs_j+n+l=0)\subset Z(B_\bff)\]
for $0\le l<r$. Since $f$ is essential, one observes that the component $(\sum_{j=1}^rs_j+n+l=0)$ is supported at the origin of $\bbC^n$, that is, it is not a component of $Z(B_f)$ over neighborhood away from the origin. 

In general, we choose a dense edge and set 
$$\bff_W=(f_j)_{j\in J(W,\bff)} \textup{ and } f_W=\prod_{j\in J(W,\bff)}f_j.$$
We then consider the complete factorization $\bff_W$ on $X/W$. Since $f_W$ is central, essential and irreducible on $X/W$,
\[(\sum_{j\in J(W,\bff)}s_j+\textup{rank}(W)+l=0)\subset Z(B_{\bff_W})\]
for $0\le l<|J(W,\bff)|$ and each component $(\sum_{j\in J(W,\bff)}s_j+\textup{rank}(W)+l=0)$ is supported on $W$. Moreover, by definition one can see that $B_\bff$ and $B_{\bff_W}$ are the same over $X\setminus \cup_{D_j\not\supset W}D_j$. Since the component $(\sum_{j\in J(W,\bff)}s_j+\textup{rank}(W)+l=0)$ of $Z(B_{\bff_W})$ is supported on $W$, we have 
\be\label{eq:aabbcc11}
(\sum_{j\in J(W,\bff)}s_j+\textup{rank}(W)+l=0)\subset Z(B_{\bff})
\ee
for $0\le l<|J(W,\bff)|$.  
\end{proof}

\subsection{Chacracteristic cycles for hyperplane arrangements}\label{subsec:5.4}
Suppose that $\bff$ is a complete factorization of a central essential irreducible hyperplane arrangement $f$ on $X=\bbC^n$ with $\supp(f=0)=D$. The canonical log resolution of $(X,D)$ is the morphism $\mu\colon Y\to X$ obtained by composition of the blowups along (the proper transform of) the union of dense edges in the increasing order of dimensions of dense edges. The canonical log resolution $\mu$ is a log resolution \cite[Theorem 3.1]{STV}. We write 
\[\tilde\bff=\mu^*\bff, \tilde f=\mu^*f \textup{ and } \widetilde j\colon U=X\setminus D\hookrightarrow Y.\]
We denote 
\[\supp(\mu^{*}D)=\sum_{i\in S} E_i\textup{ and } E_I=\bigcap_{i\in I} E_i \textup{ for } I\subseteq S.\]
We set $q$ to be the prime ideal generated by $\sum_{j=1}^r s_j+1$ in $\bbC[\bs]$.

Now, we study the relative characteristic cycle $\CC^\rel(\cM_\bff)$.
\begin{lemma}\label{lm:upstairmf}
With notations as above, we have:
\begin{enumerate}
    \item $\widetilde j_*(\sO_U[\bs]{\tilde f}^\bs)_q\simeq\shD_Y[\bs]_q{\tilde f}^\bs$;
    \item $\widetilde j_!(\sO_U[\bs]{\tilde f}^\bs)_q\simeq\shD_Y[\bs]_q{\tilde f}^{\bs+\mathbf1_r}$;
    \item for a general point $\bal=(\alpha_1,\alpha_2,\dots,\alpha_r)\in(\sum_{j=1}^r s_j+1=0)\subset \bbC^r$, the $\shD_Y$-module $\cM_{\tilde\bff,m}$ is supported on $E$ and the $\shD_Y$-module $\iota^*(\cM_{\tilde\bff,m})$ are holonomic and supported on $E$ with
    \[\CC(\iota^*(\cM_{\tilde\bff,m}))=\sum_{E_I\subseteq E} T^*_{E_I}Y,\]
    where $E$ is the exceptional divisor over the origin $0\in X=\bbC^n$, $m$ is the maximal ideal of $\bal$ in $\bbC[\bs]$ and $\iota\colon \bbC\hookrightarrow \bbC^r$ the closed embedding defined by 
    \[s\mapsto (s,\alpha_2,\alpha_3,\dots,\alpha_r);\]
    \item $\CC^\rel(\cM_{\tilde \bff,q})=\sum_{E_I\subseteq E} T^*_{E_I}Y\times (\sum_{j=1}^r s_j+1=0).$ 
\end{enumerate}
\end{lemma}
\begin{proof}
Since $\mu^* D$ is normal crossing, $\shD_X[\bs]\tilde\bff^\bs$ is Cohen-Macaulay and hence $\cM_{\tilde\bff}$ is also Cohen-Macaulay by Lemma \ref{lm:sescm} and we can calculate Bernstein-Sato ideals for $\tilde\bff$ locally around any analytic neighborhood.
Furthermore, by local computation, one can see that $B_{\tilde\bff}$ (on an analytic neighborhood) is principal, that its generator is reduced, and that $\sum_{j=1}^r s_j+l$ is a factor of the generator only when $1\le l\le r.$ Therefore, by substitution
\be\label{eq:allzeroY}
\cM_{\tilde\bff,q}^{\bk_r, \bk_r-\mathbf1_r }=\frac{\shD_Y[\bs]_q\tilde\bff^{\bs-\bk_r}}{\shD_Y[\bs]_q\tilde\bff^{\bs-\bk_r+\mathbf1_r}}=0
\ee
for $k\not= 0$, where $\bk_r=(\underbrace{k,k,\dots,k}_r)$. Hence, Part (1) follows. Using the inclusion 
\[\widetilde j_*(\sO_U[\bs]{\tilde f}^\bs)_q\hookrightarrow \widetilde j_!(\sO_U[\bs]{\tilde f}^\bs)_q\]
in Theorem \ref{thm:j_!loc} and duality, we see that $j_!(\sO_U[\bs]{\tilde f}^\bs)_q$ is the minimal extension of $\sO_U[\bs]_q\tilde \bff^\bs$. By minimality and \eqref{eq:allzeroY} for $k<0$, Part (2) follows.

Now we prove Part (3) and Part (4). We pick a general point $\bal$ on $(\sum_{j=1}^r s_j+1=0)$ and write by $m$ its maximal ideal. Since $\bal$ is general, we can assume $\bal$ is away from the components of $Z(B_{\tilde\bff})$ other than $(\sum_{j=1}^rs_j+1=0)$ (with the local $B_{\tilde\bff}$) and hence $\cM_{\tilde\bff,m}$ is supported on $E$.
We now consider the embedding
\[\iota\colon \bbC=\Spec \bbC[s]\hookrightarrow \bbC^r\]
defined by 
\[s\mapsto (s,\alpha_2,\alpha_3,\dots,\alpha_r).\]
Since $\cM_{\tilde\bff}$ is Cohen-Macaulay, one can apply the method proving \eqref{eq:inductivechr} in the proof of Theorem \ref{thm:diagpbcm} (inductively using $s_i-\alpha_i=0$ to cut $\bbC^r$ for $i\ge 2$) and obtain
\be\label{eq:projCCrel}
\L\iota^*(\cM_{\tilde\bff,m})\stackrel{q.i.}{\simeq}\iota^*(\cM_{\tilde\bff,m})\textup{ and }\CC^\rel(\iota^*(\cM_{\tilde\bff,m}))=\iota^{*}(\CC^\rel(\cM_{\tilde\bff,m})).
\ee
Since $\mu^*D$ is normal crossing, we can calculate the relative characteristic cycles of $\cM_{\tilde\bff}$ and $\cM_{\tilde\bff,m}$ explicitly and obtain 
\[
\CC^\rel(\cM_{\tilde \bff,m})=\sum_{E_I\subseteq E} T^*_{E_I}Y\times (\sum_{j=1}^r s_j+1=0)\subseteq T^*X\times \Spec\bbC[\bs]_m.
\]
In particular, Part (4) follows.
We also obtain
\[\CC^\rel(\iota^*(\cM_{\tilde \bff,m}))=\sum_{E_I\subseteq E} T^*_{E_I}Y\times (s-\alpha_1)\subseteq T^*X\times \Spec\bbC[s]_{\bar m}.\]
where $\bar m$ is the maximal ideal of $\alpha_1\in\bbC$. Thanks to Lemma \ref{lm:suppzbf}, by considering its Bernstein-Sato ideal over $\bbC[\bs]_{\bar m}$, $\iota^*(\cM_{\tilde \bff,m})$ is annihilated by $(s-\alpha_1)^m$ for some $m\ge 1$. Hence, $\iota^*(\cM_{\tilde \bff,m})$ is coherent over $\shD_Y$ and thus a good filtration of $\iota^*(\cM_{\tilde \bff,m})$ over $\shD_Y$ is also good over $\shD_Y[s]$. We further conclude 
\[\CC(\iota^*(\cM_{\tilde \bff,m}))=\sum_{E_I\subseteq E} T^*_{E_I}Y\subset T^*Y\]
and that $\iota^*(\cM_{\tilde \bff,m})$ is a holonomic $\shD_Y$-module (it is indeed regular holonomic). 

\end{proof}

\begin{lemma}\label{lm:mupfj*}
For every $\bal\in\bbC^r$, we have 
\begin{enumerate}
    \item $\mu_+(\widetilde j_*(\sO_U[\bs]{\tilde f}^\bs)_m)\stackrel{q.i.}{\simeq} j_*(\sO_U[\bs]{f}^\bs)_m$
    \item $\mu_+(\widetilde j_!(\sO_U[\bs]{\tilde f}^\bs)_m)\stackrel{q.i.}{\simeq} j_!(\sO_U[\bs]{f}^\bs)_m$
    \item $\mu_+(\cM_{\tilde\bff,m})\stackrel{q.i.}{\simeq} \frac{j_*(\sO_U[\bs]{f}^\bs)_m}{j_!(\sO_U[\bs]{f}^\bs)_m},$
    where $m$ is the maximal ideal of $\bal$ in $\bbC[\bs]$.
\end{enumerate}
\end{lemma}
\begin{proof}
Since $\mu$ is identical over $U$, Part (1) is obvious. 

By Theorem \ref{thm:pshfdualcmm}, we know that $\D$ and $\mu_+$ commute. Thus, Part (2) follows from the definition of $\widetilde j_!$ (cf. \S\ref{subsec:j_*j_!}).

Since $\mu_+$ is an exact derived functor, by Lemma \ref{lm:upstairmf} (1) and (2) we have a distinguish triangle
\[\mu_+(\widetilde j_!(\sO_U[\bs]{\tilde f}^\bs)_q)\to \mu_+(\widetilde j_*(\sO_U[\bs]{\tilde f}^\bs)_q)\to \mu_+(\cM_{\tilde\bff,q})\xrightarrow{+1}.\]
Taking the associated long exact sequence of cohomology sheaves, by Part (1) and Part (2) that we have just proved, we obtain a long exact sequence 
\[0\to \cH^{-1}(\mu_+(\cM_{\tilde\bff,q}))\rightarrow j_!(\sO_U[\bs]{\tilde f}^\bs)_q\rightarrow j_*(\sO_U[\bs]{\tilde f}^\bs)_q\rightarrow \cH^{0}(\mu_+(\cM_{\tilde\bff,q})\to0.\]
By Theorem \ref{thm:j_!loc},
\[j_!(\sO_U[\bs]{\tilde f}^\bs)_q\rightarrow j_*(\sO_U[\bs]{\tilde f}^\bs)_q\]
is injective. Thus, Part (3) follows. 
\end{proof}
\begin{proof}[Proof of Theorem \ref{thm:mainmchr}]
We first prove the case that $f$ is central essential and irreducible. By substitution, it is enough to assume $l=1$. We pick a general point $\bal\in (\sum_{j=1}^r s_j+1=0)$, and write $m$ the maximal ideal of $\bal$ in $\bbC[\bs]$. By Lemma \ref{lm:mchyp}, $(\sum_{j=1}^rs_j+1=0)$ is a component of $Z(B_\bff^{\bk_r,-\bk_r})$ for all $k\gg1$. By Theorem \ref{thm:j_*loc} and Theorem \ref{thm:j_!loc}, we have 
\[\cM_{\bff,q}^{\bk_r,-\bk_r}=\dfrac{\shD_X[\bs]_q\bff^{\bs-\bk_r}}{\shD_X[\bs]_q\bff^{\bs+\bk_r}}=\dfrac{j_*(\sO_U[\bs]{f}^\bs)_q}{j_!(\sO_U[\bs]{f}^\bs)_q}\]
and 
\[\cM_{\bff,m}^{\bk_r,-\bk_r}=\dfrac{\shD_X[\bs]_m\bff^{\bs-\bk_r}}{\shD_X[\bs]_m\bff^{\bs+\bk_r}}=\dfrac{j_*(\sO_U[\bs]{f}^\bs)_m}{j_!(\sO_U[\bs]{f}^\bs)_m}\]
for all $k\gg1$. Since $\bbC[\bs]_q$ is a DVR and hence a PID, $B_{\bff,q}^{\bk_r,-\bk_r}$ is generated by $(\sum_{j=1}^rs_j+1=0)^m$ for some integer $m\ge 1$. 
Upstairs, $\cM_{\tilde\bff,m}$ is supported on $E$ as a sheaf on $Y$ by Lemma \ref{lm:upstairmf}(3). By Lemma \ref{lm:mupfj*}(3), $\dfrac{j_*(\sO_U[\bs]{f}^\bs)_m}{j_!(\sO_U[\bs]{f}^\bs)_m}$ is thus supported at the origin $0\in X$ and hence so are $\cM_{\bff,m}^{\bk_r,-\bk_r}$ for all $k\gg 1$. Therefore, by Lemma \ref{lm:suppzbf} and relative holonomicity, 
\[T^*_{\{0\}}X\times (\sum_{j=1}^rs_j+1=0)\]
is a component of $\Ch^\rel(\cM_{\bff}^{\bk_r,-\bk_r})$ for every $k\gg 1$. We assume that its multiplicity is $\ell$. Then since relative characteristic cycles localize, 
\[\CC^\rel(\cM_{\bff,m}^{\bk_r,-\bk_r})=\CC^\rel(\dfrac{j_*(\sO_U[\bs]{f}^\bs)_m}{j_!(\sO_U[\bs]{f}^\bs)_m})=\ell\cdot T^*_{\{0\}}X\times (\sum_{j=1}^rs_j+1=0).\]
Combining Theorem \ref{thm:j_*loc}, Theorem \ref{thm:j_!loc} and Lemma \ref{lm:sescm}, we see that $\dfrac{j_*(\sO_U[\bs]{f}^\bs)_m}{j_!(\sO_U[\bs]{f}^\bs)_m}$ is Cohen-Macaulay. Similar to \eqref{eq:projCCrel}, we obtain
\[\L\iota^*(\cM_{\bff,m}^{\bk_r,-\bk_r})\stackrel{q.i.}{\simeq}\iota^*(\cM_{\bff,m}^{\bk_r,-\bk_r})\textup{ and }\CC^\rel(\iota^*(\cM_{\bff,m}^{\bk_r,-\bk_r}))=\iota^{*}(\CC^\rel(\cM_{\bff,m}^{\bk_r,-\bk_r}))\]
for $k\gg1$. Similar to the way we prove that $\iota^*(\cM_{\tilde\bff,m})$ is coherent over $\shD_Y$ in the proof of Lemma \ref{lm:upstairmf} (3), we can also obtain that $\iota^*(\cM_{\bff,m}^{\bk_r,-\bk_r})$ is coherent over $\shD_X$ for every $k\gg1$. Therefore,
\[\CC(\iota^*(\cM_{\bff,m}^{\bk_r,-\bk_r}))=\ell\cdot T^*_{\{0\}}X.\]

We now consider the following commutative diagram
\[
\begin{tikzcd}
E \arrow[r,"\eta_E"]\arrow[d,"\mu^E"] & Y\arrow[d,"\mu"] \\
\{0\}\arrow[r,"\eta"] & X.
\end{tikzcd}
\]
By Lemma \ref{lm:upstairmf}(3), $\iota*(\cM_{\tilde\bff,m})$ is a holonomic $\shD_Y$-module supported on $E$, by Kaishiwara's equivalence (see for instance \cite[Theorem 1.6.1]{HTT}),
\[\iota^*(\cM_{\tilde\bff,m})\simeq \eta_{E,+}\cN_E, \textup{ and } \CC(\cN_E)=\sum_{E_I\subseteq E} T^*_{E_I}E\subset T^*E.\]
for some holonomic $\shD_E$-module $\cN_E$. Similarly, 
\[\iota^*(\cM_{\bff,m}^{\bk_r,-\bk_r})\simeq \eta_+\cN_{\{0\}}\]
for some $\bbC$-vector space $\cN_{\{0\}}$ of dimension $\ell$. Since 
\[\L\iota^*(\cM_{\bff,m}^{\bk_r,-\bk_r})\stackrel{q.i.}{\simeq}\iota^*(\cM_{\bff,m}^{\bk_r,-\bk_r})\textup{ and } \L\iota^*(\cM_{\tilde\bff,m})\stackrel{q.i.}{\simeq}\iota^*(\cM_{\tilde\bff,m}),\]
by Proposition \ref{prop:relbasechange} and Lemma \ref{lm:mupfj*}(3), we have 
\[\iota^*(\cM_{\bff,m}^{\bk_r,-\bk_r})\simeq \mu_+(\iota^*(\cM_{\tilde\bff,m}))\]
for $k\gg 1$ and hence 
\be\label{eq:muEqi}
\mu^E_+(\cN_E)\stackrel{q.i.}\simeq \cN_{\{0\}}.
\ee
We now apply the Dubson-Kashiwara index theorem (see for instance \cite[Theorem 9.1]{Gil} and also \cite[Theorem 1.6]{WZ} in the log case) and obtain that 
\[\sum_{i}(-1)^ih^i(\mu_+^E(\cN_{E)}))=T^*_{E}E\cdot\sum_{E_I\subseteq E} T^*_{E_I}E\]
where $h^i$ denote the dimension of the $i$-th cohomology and the intersection number on the right hand side is the degree of the zero cycle in $T^*_EE\simeq E$.
It is well known that 
\[T^*_EE\cdot\sum_{E_I\subseteq E} T^*_{E_I}E=(-1)^{n-1}\chi(E^o),\]
where $E^o=E\setminus \bigcup_{E_I\subsetneq E}E_I$. By the construction of $\mu$, 
\[E^o= \P(X)\setminus \bigcup_{H\in D^{\{0\}}}\P(H).\]
By \eqref{eq:muEqi}, we then know the dimension of $\cN_{\{0\}}$ and hence $\ell$ are both $$(-1)^{n-1}\chi(\P(X)\setminus \bigcup_{H\in D^{\{0\}}}\P(H)).$$

We thus have proved the case that $f$ is central essential and irreducible. In general, one can pick a dense edge $W$ and replace $f$ by $f_W$ and $X$ by $X/W$. Since $X\to X/W$ is smooth, we get $T^*_WX$ by pulling back $T^*_{\{0\}}(X/W)$ and the general case follows.  
\end{proof}

\bibliographystyle{amsalpha}
\bibliography{mybib}
\end{document}